\documentclass[aps,pre,12pt,a4paper,eqsecnum,leqno,preprint,showkeys]{revtex4-1}%%% eqno resets every section 
\usepackage{amssymb, amsmath, amsfonts, amsthm}%%% option packages
\usepackage{graphics}%%% option packages
\usepackage[dvips]{graphicx}%%% option packages
\usepackage{eepic}%%% option packages
\usepackage{epic}%%% option packages
\usepackage{tikz}%%% package to produce figures using tikz
\newcommand{\intd}{\, \textnormal{d}}

\newcommand{\sgn}{\textnormal{sign}}
\newcommand{\mb}[1]{\mathbf{#1}}
\newcommand{\mbb}[1]{\mathbb{#1}}

\newcommand{\mcl}[1]{\mathcal{#1}}
\newcommand{\mnm}[1]{\textnormal{#1}}
\newcommand{\mfk}[1]{\mathfrak{#1}}
\def\Xint#1{\mathchoice
{\XXint\displaystyle\textstyle{#1}}%
{\XXint\textstyle\scriptstyle{#1}}%
{\XXint\scriptstyle\scriptscriptstyle{#1}}%
{\XXint\scriptscriptstyle\scriptscriptstyle{#1}}%
\!\int}
\def\XXint#1#2#3{{\setbox0=\hbox{$#1{#2#3}{\int}$}
\vcenter{\hbox{$#2#3$}}\kern-.5\wd0}}

\def\dashint{\Xint-}
\newtheorem{thm}{Theorem}[section]
\newtheorem{crl}[thm]{Corollary}

\newtheorem{lmm}[thm]{Lemma}

\newtheorem{prp}[thm]{Proposition}

\theoremstyle{definition}

\theoremstyle{remark}
\begin{document}

\title{Characterizations of the hydrodynamic limit of the Dyson model}

%{RIMS K\^oky\^uroku Bessatsu sample}

%\TitleHead{Characterizations of the hydrodynamic limit of the Dyson model}          %optional
%\dedicatory{Here is a Dedication}           %optional

%\date{\today}

\author{Sergio Andraus}
\email{andraus@spin.phys.s.u-tokyo.ac.jp}
\affiliation{Department of Physics, Graduate School of Science, The University of Tokyo, 7-3-1 Hongo, Bunkyo-ku, Tokyo 113-0033, Japan.}

\author{Makoto Katori}
\email{katori@phys.chuo-u.ac.jp}
%\thanks{}  %optional
\affiliation{Department of Physics, Graduate School of Science and Engineering, Chuo University, 1-13-27 Kasuga, Bunkyo-ku, Tokyo 112-8551, Japan.}
%\AuthorHead{S. Andraus, M. Katori}         %optional but recommended
%\pacs{MACS2010 classification: 82C22, 15B52, 35Q31, 44A15.}

         %optional

\begin{abstract}      %optional
Under appropriate conditions for the initial configuration,
the empirical measure of the $N$-particle Dyson model
with parameter $\beta \geq 1$ converges to a unique
measure-valued process as $N$ goes to infinity,
which is independent of $\beta$.
The limit process is characterized by its Stieltjes transform
called the Green's function.
Since the Green's function satisfies the complex Burgers
equation in the inviscid limit,
this is called the hydrodynamic limit of the Dyson model.
We review the relations among the hydrodynamic equation
of the Green's function, 
the continuity equation of the probability density function,
and the functional equation of the Green's function.
The basic tools to prove the relations are
the Hilbert transform, a special case of the Sokhotski-Plemelj
theorem,
and the method of characteristics for solving 
partial differential equations.
For two special initial configurations,
we demonstrate how to characterize the limit processes
using these relations.
\end{abstract}

%
% The text goes here.  
% Be sure to use the appropriate "theorem-like" environment as 
% is the following examples.  Never use plain TeX commands for these, as
% they will cause interference with the styles of other papers. 

\maketitle

%\tableofcontents      %optional

%%%%%%%%%%%%%%%%%%%%%%%%%%%%%%%%%%%%%%%%%%%%%%%%
\section{Introduction}\label{sec:intro}
%%%%%%%%%%%%%%%%%%%%%%%%%%%%%%%%%%%%%%%%%%%%%%%%
For $N \in \mbb{N}$, let
$\mb{B}(t)=(B_1(t), B_2(t), \dots, B_N(t)), t \geq 0$ be an $N$-dimensional
standard Brownian motion in a probability space $(\Omega, \mnm{P})$
with a filtration 
$\mcl{F}=\{\mcl{F}_t : t \geq 0\}$.
Let $\mbb{W}_N$ denote the Weyl chamber of type A,
$$
\mbb{W}_N= \{ \mb{x}=(x_1, x_2, \dots, x_N) \in \mbb{R}^N :
x_1 < x_2 < \cdots < x_N \}
$$
with closure $\overline{\mbb{W}}_N$.
Dyson's Brownian motion model of $N$ particles 
with parameter $\beta \geq 1$,
$\mb{X}^N(t)=(X^N_1(t), X^N_2(t), \dots, X^N_N(t)), t \geq 0$,
is defined as the solution of the following system of
stochastic differential equations (SDEs) \cite{dyson62},
$$
dX^N_i(t)=dB_i(t)+\frac{\beta}{2} \sum_{\substack{1 \leq j \leq N, \cr j \not=i}}
\frac{dt}{X^N_i(t)-X^N_j(t)}, \quad t \geq 0,
\quad 1 \leq i \leq N,
$$
started at $\mb{x}^N \in \overline{\mbb{W}}_N$.
Here we perform the time change
$$
\frac{\beta N}{2}t \to t
$$
and consider the situation such that the SDEs are given as
\begin{equation}
dX^N_i(t)=\sqrt{\frac{2}{\beta N}} dB_i(t)
+\frac{1}{N} \sum_{\substack{1 \leq j \leq N, \cr j \not=i}}
\frac{dt}{X^N_i(t)-X^N_j(t)}, \quad t \geq 0,
\quad 1 \leq i \leq N,
\label{eq:Dyson2}
\end{equation}
with $\mb{X}^N(0)=\mb{x}^N \in \overline{\mbb{W}}_N$. 
In this paper we simply call this system
of interacting Brownian motions on $\mbb{R}$ 
{\it the Dyson model} \cite{AGZ10,katori16}.

Let $\mfk{M}$ be the space of probability measures on $\mbb{R}$
equipped with its weak topology.
For $T > 0$, $\mnm{C}([0, T] \to \mfk{M})$ denotes the
space of continuous processes defined in the time period $[0, T]$ 
realized in $\mfk{M}$.
We regard the empirical measure
of the solution $\mb{X}^N(t)=(X^N_1(t), \dots, X^N_N(t))$ of \eqref{eq:Dyson2}, 
\begin{equation}
\Xi^N(t, \cdot)=\frac{1}{N} \sum_{i=1}^N \delta_{X^N_i(t)}(\cdot), \quad t \in [0, T], 
\label{eq:XiN}
\end{equation}
as an element of $\mnm{C}([0, T] \to \mfk{M})$.
We express its initial configuration by
\begin{equation}
\xi^N(\cdot) \equiv \Xi^N(0, \cdot)=\frac{1}{N} \sum_{i=1}^N \delta_{x^N_i}(\cdot) \in \mfk{M}.
\label{eq:xiN}
\end{equation}
The following is proved.

%%%%%%%%%%%%%%%%%%%%%%%%%%%%%%%%%%%%%%%
\begin{thm}[Anderson, Guionnet, Zeitouni {\cite[Proposition 4.3.10]{AGZ10}}]
\label{thm:AGZ1}

Let $(\mb{x}^N)_{N \in \mbb{N}}$ be a sequence of initial configurations such that
$\mb{x}^N \in \overline{\mbb{W}}_N$, 
$$
\sup_{N \geq 0} \frac{1}{N} \sum_{i=1}^N \log \{(x^N_i)^2+1\} < \infty,
$$
and $\xi^N(\cdot)$ converges weakly to a measure $\mu(\cdot) \in \mfk{M}$
as $N \to \infty$. 
Then for any fixed $T < \infty$,
\begin{equation}
(\Xi^N(t, \cdot))_{t \in [0, T]}
\Longrightarrow ^{\exists !} (\mu(t, \cdot))_{t \in [0, T]}
\quad \mbox{a.s. in $\mnm{C}([0, T] \to \mfk{M})$},
\label{eq:converge1}
\end{equation}
where $\mu(0, \cdot) =\mu(\cdot)$ and the function
\begin{equation}
G(t, z)=\int_{\mbb{R}} \frac{d\mu(t, x)}{z-x}
\label{eq:G1}
\end{equation}
satisfies the equation
\begin{equation}
\frac{\partial G(t,z)}{\partial t} 
+G(t,z) \frac{\partial G(t,z)}{\partial z} =0,
\quad t \in [0, T], 
\quad z \in \mbb{C} \setminus \mbb{R}.
\label{eq:G2}
\end{equation}
\end{thm}
%%%%%%%%%%%%%%%%%%%%%%%%%%%%%%%%
\vskip 0.3cm

The function $(G(t, \cdot))_{t \in [0, T]}$ defined by 
the {\it Stieltjes transform} \eqref{eq:G1} is called
the {\it Green's function} (or the {\it resolvent}) 
for the measure-valued process $(\mu(t, \cdot))_{t \in [0, T]}$.
The equation \eqref{eq:G2} can be regarded as the
{\it complex Burgers equation} in the inviscid limit
({\it i.e.}, the (complex) one-dimensional Euler equation).
Thus the $N \to \infty$ limit given by this theorem
is called the {\it hydrodynamic limit of the Dyson model}
\cite{AGZ10,blaizotnowak10,forrestergrela15}. 

For $x \in \mbb{R}$, assume that a function $F$ is bounded and
integrable over $(-\infty, x-\varepsilon]$ and $[x+\varepsilon, \infty)$
for any $\varepsilon >0$. 
We introduce the Cauchy principal value at $x$ for the integral of $F$ over 
$\mbb{R} \setminus \{x\}$ as
$$
\dashint_{\mbb{R} \setminus \{x\}} F(y) d y
:=\lim_{\epsilon \downarrow 0}
\left\{ 
\int_{-\infty}^{x-\varepsilon} F(y) d y
+\int^{\infty}_{x+\varepsilon} F(y) d y \right\},
$$
if the limit exists and is finite. 
For $f \in L^p(\mbb{R}), 1 < p < \infty$, the Hilbert transform is defined by 
\begin{equation}
\mcl{H}[f](x):=\frac{1}{\pi}\dashint_{\mbb{R} \setminus \{x\}} 
\frac{f(y)}{x-y} d y, \quad x \in \mbb{R}.
\label{eq:Hilbert1}
\end{equation}
Note that $\mcl{H}[f](x)$ is real-valued for all $x \in \mbb{R}, f \in L^p(\mbb{R})$,
$1 < p < \infty$ by definition.

In this paper we consider the case in which $(\mu(t, \cdot))_{t \in [0, T]}$
has a probability density function for any $T < \infty$. 
That is, we can write
$$
\mu(t, A)=\int_A \rho(t, x) dx,
\quad t \in [0, \infty),
$$
for any Borel set $A \in \mfk{B}(\mbb{R})$.
We assume the following conditions for $\rho$.
\begin{description}
\item[{\bf [C1]}] \quad
$\rho(t, x)$ is piecewise differentiable with respect to 
$t \in (0, \infty)$ and $x \in \mbb{R}$;
the piecewise-defined derivatives are written as
$\displaystyle{\frac{\partial \rho(t, x)}{\partial t}}$,
$\displaystyle{\frac{\partial \rho(t, x)}{\partial x}}$.
\item[{\bf [C2]}] \quad
$\rho(t, \cdot)$, $\displaystyle{\frac{\partial \rho(t, \cdot)}{\partial t}}$,
$\displaystyle{\frac{\partial \rho(t, \cdot)}{\partial x} \in L^p(\mbb{R}), 1 < p < \infty}$,
for each $t \in (0, \infty)$.
Thus, their Hilbert transforms are well-defined.
\end{description}
From now on, for piecewise differentiable functions, 
the derivatives are assumed to be piecewise-defined.

One of the purposes of the present paper is to give
a precise proof to the following statement.

%%%%%%%%%%%%%%%%%%%%%%%%%%%%%%%%%%%%%%%%%%%
\begin{thm}
\label{thm:main}
Assume that the Green's function for the process
$(\mu(t, \cdot))_{t \in [0, \infty)}$
is given by
\begin{equation}
G(t, z)=\int_{\mbb{R}} \frac{\rho(t,x)}{z-x} dx,
\quad t \in [0, \infty), \quad z \in \mbb{C} \setminus \mbb{R}, 
\label{eq:G3}
\end{equation}
where
$\rho$ satisfies the conditions {\bf [C1]} and {\bf [C2]}. 
Then the partial differential equation \eqref{eq:G2} for $G$
leads to the following two characterizations for the process.
\begin{description}
\item[{\rm (i)}] \quad
The probability density function $\rho$ satisfies the following equation
\begin{equation}
\frac{\partial \rho(t, x)}{\partial t} 
+ \pi \frac{\partial}{\partial x} \Big\{
\rho(t, x) \mcl{H}[\rho(t, \cdot)](x) \Big\} =0.
\label{eq:rho2}
\end{equation}
\item[{\rm (ii)}] \quad
The following functional equation is solved by $G$, 
\begin{equation}
G(t, z)=G \Big(0, z-t G(t,z) \Big), \quad
t \in [0, \infty).
\label{eq:functional}
\end{equation}
\end{description}
\end{thm}
%%%%%%%%%%%%%%%%%%%%%%%%%%%%%%%%%%%%%%%%%%
\vskip 0.3cm

Equation \eqref{eq:rho2} is the continuity equation
of the density function of the system,
$$
\frac{\partial \rho(t, x)}{\partial t} 
=-\frac{\partial J(t, x)}{\partial x},
\quad t \in [0, \infty), \quad x \in \mbb{R},
$$
with the current density function
\begin{eqnarray}
J(t, x) &=& \pi \rho(t, x) \mcl{H}[\rho(t, \cdot)](x)
\nonumber\\
&=& - \rho(t, x) \frac{\partial}{\partial x}
\int_{\mbb{R}} \rho(t, y) V(x-y) dy,
\nonumber
\end{eqnarray}
associated with the logarithmic potential
$V(x)=-\log|x|$.
In other words, \eqref{eq:rho2} will provide the
reaction-diffusion equation which governs the 
macroscopic behavior of the 
{\it one-dimensional log-gas} \cite{forrester10}.
An interesting and important fact is that the formula
\begin{equation}
\rho(t, x)=-
\Im \left[ \lim_{\varepsilon \downarrow 0}
\frac{1}{\pi} G(t, x+ \sqrt{-1} \varepsilon) \right],
\quad t \in [0, \infty), \quad x \in \mbb{R},
\label{eq:rho_G}
\end{equation}
is established
(see the comment given above Proposition \ref{prp:HilbertAvgResolvent}
in the next section) and 
it is easier to obtain $\rho$ 
by solving the functional equation \eqref{eq:functional}
for $G$ and using the formula \eqref{eq:rho_G}
rather than by solving \eqref{eq:rho2}. 
Another purpose of the present paper is
to demonstrate the usefulness of the 
functional equation \eqref{eq:functional}. 

The paper is organized as follows.
In Section \ref{sec:Hilbert} we prove the formula \eqref{eq:rho_G}
and Propositions concerning the basic properties of the Hilbert transform.
Section \ref{sec:proof} is devoted to proving Theorem \ref{thm:main}.
There we will use the properties of the Hilbert transform given in
Section \ref{sec:Hilbert} in order to prove (i) of Theorem \ref{thm:main}. 
Then the {\it method of characteristics} 
\cite{AGZ10,blaizotnowak08,blaizotnowak10,forrestergrela15}
is applied to prove (ii) of Theorem \ref{thm:main}.
In Section \ref{sec:solutions} we demonstrate how to
solve the functional equation \eqref{eq:functional} for $G$
and determine $\rho$ through the formula \eqref{eq:rho_G}
for two special cases of the initial data $\mu$.
In both cases, the support of $\rho$ is bounded on $\mbb{R}$
and the conditions {\bf [C1]} and {\bf [C2]} are clearly satisfied.

%%%%%%%%%%%%%%%%%%%%%%%%%%%%%%%%%%%%%%%%%%%%%%%%
\section{Basic properties of the Hilbert transform}\label{sec:Hilbert}
%%%%%%%%%%%%%%%%%%%%%%%%%%%%%%%%%%%%%%%%%%%%%%%%

The relation \eqref{eq:rho_G} will be obtained
as the imaginary part of the upper equation in \eqref{eq:Hilbert_rho}
in the following proposition.
This proposition is a special case of the Sokhotski-Plemelj theorem.

%%%%%%%%%%%%%%%%%%%%%%%%%%%%%%%%%%%%%%%%%%%%%%%%%%%%%%%%
\begin{prp}\label{prp:HilbertAvgResolvent}
For any $t \in [0, \infty)$, 
the Hilbert transform of $\rho(t, \cdot)$ and 
the Green's function $G$ are related by
\begin{equation}
\lim_{\varepsilon\downarrow 0}
\frac{1}{\pi} G(t, x \pm \sqrt{-1} \varepsilon)
= \mcl{H} [\rho(t, \cdot)](x) \mp \sqrt{-1} \rho(t, x),
\quad x \in \mbb{R}.
\label{eq:Hilbert_rho}
\end{equation}
\end{prp}
%%%%%%%%%%%%%%%%%%%%%%%%%%%%%%%%%%%%%%%%%%%%%%%
%%%%%%%%%%%%%%%%%%%%%%%%%%%%%%%%%%%%%%%%%%%%%%%
\begin{proof}
Consider a closed, simple, and positively-oriented contour $C$ on $\mbb{C}$ 
and a complex function $f$ which is analytic on $C$. 
Denote by $D$ the open region enclosed by $C$. We define the function
$$
\phi(z):=\frac{1}{2\pi \sqrt{-1}}\int_{C} 
\frac{f(w)}{w-z} d w.
$$
Clearly, $\phi(z)$ is well-defined when $z \notin C$, 
but it is discontinuous at $z \in C$. 
Choose $\zeta \in C$ and consider the limits 
when $z$ tends to $\zeta$ from the inside and from the outside of $D$. 
Assume, furthermore, that $C$ is smooth at $\zeta$. 
In the first case, there exists an angle $\alpha$ such that
\begin{eqnarray}
&& \phi_{\mnm{in}}(\zeta):=
\lim_{\substack{z\to\zeta, \\ z \in D}}\phi(z)
\nonumber\\
&& \quad = \frac{1}{2\pi \sqrt{-1}} \dashint_{C \setminus \{\zeta\}}
\frac{f(w)}{w-\zeta} d w 
+\lim_{\varepsilon \downarrow 0} \frac{1}{2\pi \sqrt{-1}} 
\int_0^\pi \frac{f(\zeta+\varepsilon \mnm{e}^{\sqrt{-1} (\alpha+\theta)})}
{\varepsilon\mnm{e}^{\sqrt{-1}(\alpha+\theta)}}
\sqrt{-1} \varepsilon\mnm{e}^{\sqrt{-1} (\alpha+\theta)} d \theta \nonumber\\
&& \quad =\frac{1}{2\pi \sqrt{-1}} \dashint_{C \setminus \{\zeta\}} 
\frac{f(w)}{w-\zeta} d w +\frac{1}{2}f(\zeta).
\nonumber
\end{eqnarray}
The principal-value integral is taken along $C$ while excluding the point $\zeta$. 
The case where $z$ approaches $\zeta$ from the outside of $D$
can be calculated similarly:
$$
\phi_{\mnm{out}}(\zeta):=\lim_{\substack{z\to\zeta, \\ z \in \mbb{C} \setminus \overline{D}}}\phi(z)
=\frac{1}{2\pi \sqrt{-1}} \dashint_{C \setminus \{\zeta\}}
\frac{f(w)}{w-\zeta} d w -\frac{1}{2}f(\zeta).
$$
We assume that, for each fixed $t \in [0, \infty)$, the domain
of the probability density function $\rho(t, \cdot)$ can be
extended from $\mbb{R}$ into the complex plane. 
Then we specialize the above result for the case 
where $f(\cdot)=\rho(t, \cdot)$ with a fixed $t \in [0, \infty)$. 
We choose $C$ to be the contour given 
by the parameterizations $w=y$ with $y$ going from $-R$ to $R$, 
and $w=R\mnm{e}^{\sqrt{-1} \theta}$ with $\theta$ going from $0$ to $\pi$, 
while letting $R$ tend to infinity (see Figure~\ref{fig:ContourC}).
% (see Figure~\ref{fig:ContourGamma}). 
\begin{figure}
\begin{center}
\begin{tikzpicture}
\draw [->] (-3,0) -- (3,0);
\draw [->] (0,-0.1) -- (0,3);
\draw [thick] (2.5,0) arc [radius=2.5, start angle=0, end angle= 180];
\draw [thick] (-2.5,0) -- (2.5,0);
\draw (0,0) -- (1.2,2.2);
\node [above left] at (0.6,1.1) {$R$};
\node at (-1.2,0) {$>$};
\node at (1.2,0) {$>$};
\node [rotate=-45] at (1.77,1.77) {$<$};
\node [rotate=45] at (-1.77,1.77) {$<$};
\node [right] at (0,3) {Im$(w)$};
\node [below] at (3,0) {Re$(w)$};
\end{tikzpicture}
\end{center}
\caption{\label{fig:ContourC}Contour $C$ before the limit where $R\to\infty$.}
\end{figure}
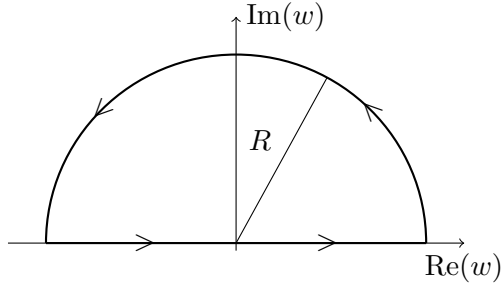
Taking $\zeta = x \in \mbb{R}$, we have
\begin{eqnarray}
\phi_{\mnm{in}}(x) &=& \frac{1}{2\pi \sqrt{-1}} \dashint_{\mbb{R} \setminus \{x\} } 
\frac{\rho(t, y)}{y-x} d y +\frac{1}{2}\rho(t, x), 
\label{eq:SokhotskiPlemeljIn} \\
\phi_{\mnm{out}}(x) &=& \frac{1}{2\pi \sqrt{-1}} \dashint_{\mbb{R} \setminus \{x\}} 
\frac{\rho(t, y)}{y-x} d y -\frac{1}{2} \rho(t, x).
\label{eq:SokhotskiPlemeljOut}
\end{eqnarray}
The reason for this is the following.
By the condition {\bf [C2]}, 
$\rho(t, x) \to 0$ as $x \to \pm \infty$, and
we will be able to extend $\rho(t, \cdot)$ to a function on
the complex upper half-plane $\mbb{C}_+ := \{z \in \mbb{C}: \Im z > 0\}$
so that $\lim_{z \to \infty, z \in \mbb{C}_+} \rho(t, z)=0$.
Hence, the part of the integral with a semi-circular contour vanishes:
$$
\lim_{R \to \infty} \int_0^{\pi} \frac{\rho(t, R \mnm{e}^{ \sqrt{-1} \theta})}{R\mnm{e}^{\sqrt{-1} \theta}}
\sqrt{-1} R\mnm{e}^{\sqrt{-1} \theta} d \theta
= \lim_{R \to \infty} \sqrt{-1} \int_0^{\pi} \rho(t, R \mnm{e}^{ \sqrt{-1} \theta} ) d \theta=0.
$$
This also implies that, for $t \in [0, \infty)$, 
\begin{equation}
2\pi \sqrt{-1} \phi(z) = \int_{\mbb{R}}\frac{\rho(t, y)}{y-z} d y = -G(t, z),
\quad z \in \mbb{C} \setminus \mbb{R}. 
\label{eq:CauchyInt1}
\end{equation}
The contour we have chosen covers the complex upper half-plane, 
so if $z$ approaches the real axis from above (resp., below), 
we must use $\phi_{\mnm{in}}(x)$ (resp., $\phi_{\mnm{out}}(x)$). Then, we obtain
$$
\lim_{\varepsilon \downarrow 0} G(t, x+\sqrt{-1} \varepsilon) = -2\pi \sqrt{-1} \phi_{\mnm{in}}(x), \quad 
\lim_{\varepsilon \downarrow 0} G(t, x-\sqrt{-1} \varepsilon)= -2\pi \sqrt{-1} \phi_{\mnm{out}}(x),
\quad x \in \mbb{R}.
$$
The result \eqref{eq:Hilbert_rho} follows by the definition \eqref{eq:Hilbert1} of the Hilbert transform $\mcl{H}$
applied to \eqref{eq:SokhotskiPlemeljIn} and \eqref{eq:SokhotskiPlemeljOut}.
\end{proof}
%%%%%%%%%%%%%%%%%%%%%%%%%%%%%%%%%%%%%%%%%%%%%%%%

Now we give the basic properties of the Hilbert transform,
which will be used to prove Theorem \ref{thm:main} in the next section.

\begin{prp}\label{prp:InverseHilbert}
For $f \in L^p(\mbb{R}), 1 < p < \infty$, then
the inverse Hilbert transform of $f$ is given by $-\mcl{H}[f]$, that is,
\begin{equation}
\mcl{H} \Big[ \mcl{H}[f] \Big](x)= - f(x).
\label{eqn:inverseH}
\end{equation}
\end{prp}
\begin{proof}
In the proof of Proposition \ref{prp:HilbertAvgResolvent}, we gave
an expression \eqref{eq:SokhotskiPlemeljIn} for
$\lim_{\varepsilon \downarrow 0} \phi(x+\sqrt{-1} \varepsilon),$ $x \in \mbb{R}$
for $\phi(z)$ given by \eqref{eq:CauchyInt1} for $z \in \mbb{C} \setminus \mbb{R}$.
This result implies the following expression for the
Hilbert transform which is different from \eqref{eq:Hilbert1}, 
\begin{equation}\label{eq:HilbertAltIntegral}
\mcl{H}[f](x)=\sqrt{-1} f(x)-
\lim_{\varepsilon \downarrow 0} \frac{1}{\pi} \int_{\mbb{R}} \frac{f(y)}{y-(x+\sqrt{-1} \varepsilon)} d y.
\end{equation}
The Hilbert transform is doubly performed using this expression as
\begin{eqnarray}
\label{eq:doubleH}
&& \mcl{H} \Big[ \mcl{H}[f] \Big](x) = -f(x)- \lim_{\varepsilon \downarrow 0} \frac{2 \sqrt{-1}}{\pi}
\int_{\mbb{R}} dy \, \frac{f(y)}{y-(x+\sqrt{-1} \varepsilon)} 
\\
&& \qquad - \lim_{\varepsilon_1 \downarrow 0} \lim_{\varepsilon_2 \downarrow 0}
\frac{1}{\pi^2}
\int_{\mbb{R}} d y \int_{\mbb R} d z 
\frac{f(z)}{[y-(x+\sqrt{-1} \varepsilon_1)][y-(z-\sqrt{-1} \varepsilon_2)]}.
\nonumber
\end{eqnarray}
In the double integral in the third term, 
we evaluate the integral over $y \in \mbb{R}$ as the integral over 
the contour $C$ depicted in Figure~\ref{fig:ContourC}, 
which encloses $\mbb{C}_+$ including a simple pole at $y=x+\sqrt{-1} \varepsilon_1$.
By Cauchy's integral formula, the third term is calculated as
\begin{eqnarray}
&& -\lim_{\varepsilon_1 \downarrow 0} \lim_{\varepsilon_2 \downarrow 0} \frac{1}{\pi^2} 
\int_{\mbb{R}} dz \frac{2 \pi \sqrt{-1} f(z)}{(x+\sqrt{-1} \varepsilon_1)-(z-\sqrt{-1} \varepsilon_2)}
\nonumber\\
&=&
\lim_{\varepsilon_1 \downarrow 0} \lim_{\varepsilon_2 \downarrow 0} \frac{2 \sqrt{-1}}{\pi} 
\int_{\mbb{R}} dz \frac{f(z)}{z-\{x+\sqrt{-1}(\varepsilon_1+\varepsilon_2)\}}.
\nonumber
\end{eqnarray}
This cancels the second term in \eqref{eq:doubleH} and \eqref{eqn:inverseH} is obtained. 
\end{proof}
%%%%%%%%%%%%%%%%%%%%%%%%%%%%%%%%%%%%%%%%%%%%%%%%%%%%%%%%%%%%%%%%%%%%

\begin{prp}\label{prp:HilbertDerivative}
Assume that $f$ has a piecewise-defined derivative
$\displaystyle{ \frac{df}{dx} }$, and 
$f$, $\displaystyle{\frac{df}{dx} \in L^p(\mbb{R})},$ $1 < p < \infty$. 
Then
$$
\frac{d \mcl{H}[f](x)}{d x}=\mcl{H} \left[ \frac{df}{dx} \right](x),
\quad x \in \mbb{R}.
$$
\end{prp}
\begin{proof}
We use the expression \eqref{eq:HilbertAltIntegral} 
for the Hilbert transform, 
\begin{eqnarray}
\mcl{H}\left[ \frac{df}{dx} \right](x) &=& \sqrt{-1} \frac{d f (x)}{d x}
-\lim_{\varepsilon \downarrow 0} \frac{1}{\pi} 
\int_{\mbb{R}} \frac{d f(y)}{d y} \frac{1}{y-(x+\sqrt{-1} \varepsilon)} d y
\nonumber\\
&=&\sqrt{-1} \frac{df(x)}{dx}
+\lim_{\varepsilon \downarrow 0} \frac{1}{\pi} \int_{\mbb{R}} f(y) 
\frac{d}{dy} \left[\frac{1}{y-(x+\sqrt{-1} \varepsilon)} \right] d y. 
\nonumber
\end{eqnarray}
The second equality is obtained from an integration by parts, 
where we use the assumption that $f\in L^p(\mbb{R})$, so $f$ vanishes at infinity. 
We change the variable of differentiation from $y$ to $x$ inside the integral.
Then the above is equal to
$$
\sqrt{-1} \frac{df(x)}{dx} 
-\frac{d}{dx} \lim_{\varepsilon \downarrow 0} 
\frac{1}{\pi} \int_{\mbb{R}} \frac{f(y)}{y-(x+\sqrt{-1} \varepsilon) } d y
= \frac{d \mcl{H}[f](x)}{dx},
$$
where \eqref{eq:HilbertAltIntegral} is again used.
Then the proof is completed.
\end{proof}

For the following proposition, it will be useful to know how the Hilbert transform 
behaves when a Fourier transform is present. 
Here we rewrite the Hilbert transform as
$$
\mcl{H}[f](x)=[f\star g](x),
$$
where $\star$ denotes the convolution product
$$
[f\star g](x):=\int_{\mbb{R}}f(y)g(x-y)\intd y,
$$
and $g(x)= 1/(\pi x)$. 
The integral is interpreted as a principal value where necessary. 
It is clear that the Fourier transform defined by
$$
\mcl{F}[f](\nu):=\int_{\mbb{R}}f(x)\mnm{e}^{-2\pi \sqrt{-1} x\nu} d x
$$
transforms the Hilbert transform of $f$ into
$\mcl{F} \Big[ \mcl{H}[f] \Big](\nu)=\Gamma(\nu)\Phi(\nu)$, 
where $\Phi(\nu)$ and $\Gamma(\nu)$ are the Fourier transforms of $f(x)$ and $g(x)=1/(\pi x)$, respectively. 
It is easy to verify that $\Gamma(\nu)$ is given by
\begin{equation}\label{eq:GammaSgn}
\Gamma(\nu)=-\sqrt{-1} \ \sgn(\nu)
:= \left\{ \begin{array}{ll}
\sqrt{-1}, \quad & \mbox{if $\nu < 0$}, \cr
0, \quad & \mbox{if $\nu=0$}, \cr
-\sqrt{-1}, \quad & \mbox{if $\nu > 0$}.
\end{array} \right.
\end{equation}
In summary, with \eqref{eq:GammaSgn}, we obtain the formula
\begin{equation}
\mcl{F} \Big[\mcl{H}[f] \Big](\nu)=\Gamma(\nu) \mcl{F}[f](\nu).
\label{eq:F_H}
\end{equation}

We will proceed with the proof of the following proposition. 

%%%%%%%%%%%%%%%%%%%%%%%%%%%%%%%%%%%%%%%%%%%%%%%%%%%
\begin{prp}[Carton-Lebrun {\cite{cartonlebrun77}}]
\label{prp:HilbertProduct}
Assume that $f \in L^p(\mbb{R})$, $h \in L^q(\mbb{R})$ with
$1 < p < \infty,$ $1 < q < \infty$ and $1/p+1/q \leq 1$. 
Then
\begin{equation}\label{eq:HilbertProduct}
\mcl{H}[f](x)\mcl{H}[h](x) - f(x)h(x)=\mcl{H} \Big[ f\mcl{H}[h]+\mcl{H}[f]h \Big](x)
\quad \mbox{a.e.}
\end{equation}
\end{prp}
%%%%%%%%%%%%%%%%%%%%%%%%%%%%%%%%%%%%%%%%%%%%%%%%%%%%
\vskip 0.3cm

Before the proof of this statement, we establish a helpful equation \cite{cartonlebrun77}.
%%%%%%%%%%%%%%%%%%%%%%%%%%%%%%%%%%%%%%%%%%%%%%
\begin{lmm}\label{lmm:auxiliary}
Assume that $\Phi, \Theta \in L^2(\mbb{R})$. Then
with $\Gamma(x)$ given by \eqref{eq:GammaSgn}, 
$$
[\Gamma\Phi\star\Gamma\Theta](x)-[\Phi\star\Theta](x)
=\Gamma(x) \Big\{ [\Phi\star\Gamma\Theta](x)+[\Gamma\Phi\star\Theta](x) \Big\}.
$$
\end{lmm}
%%%%%%%%%%%%%%%%%%%%%%%%%%%%%%%%%%%%%%%%%%%%
\begin{proof}
We perform a straightforward calculation. The first part of the LHS gives
$$
[\Gamma\Phi\star\Gamma\Theta](x)=-\int_{\mbb{R}}\sgn[y(x-y)]\Phi(y)\Theta(x-y) d y,
$$
so the LHS gives
$$
[\Gamma\Phi\star\Gamma\Theta](x)-[\Phi\star\Theta](x)
=-2\sgn(x) \int_0^x \Phi(y)\Theta(x-y) d y.
$$
For the RHS, we obtain
\begin{eqnarray}
&& \Gamma(x) \Big\{[\Phi\star\Gamma\Theta](x)+[\Gamma\Phi\star\Theta](x) \Big\}
\nonumber\\
&& \quad = -\sgn(x) \int_{\mbb{R}} [\sgn(y)+\sgn(x-y)] \Phi(y)\Theta(x-y) d y
\nonumber\\
&& \quad = -2\sgn(x) \int_{0}^{x} \Phi(y)\Theta(x-y) d y,
\nonumber
\end{eqnarray}
which is identical to the LHS, as desired.
\end{proof}
%%%%%%%%%%%%%%%%%%%%%%%%%%%%%%%%%%%%%%%%%%

\begin{proof}[Proof of Proposition~\ref{prp:HilbertProduct}]
Here we consider the case $p=q=2$. 
We take the Fourier transform of 
\eqref{eq:HilbertProduct}. 
Set $\mcl{F}[f](\nu)=\Phi(\nu)$ and $\mcl{F}[h](\nu)=\Theta(\nu)$. 
By \eqref{eq:F_H} the LHS gives
$$
\mcl{F} \Big[ \mcl{H}[f]\mcl{H}[h] - fh \Big](\nu)=[\Gamma\Phi\star\Gamma\Theta](\nu)-[\Phi\star\Theta](\nu).
$$
By virtue of Lemma \ref{lmm:auxiliary}, this is equal to
$\Gamma(\nu) \{ [\Phi\star\Gamma\Theta](\nu)+[\Gamma\Phi\star\Theta](\nu) \}$.
Again using \eqref{eq:F_H}, this is rewritten as
$$
\Gamma(\nu)\mcl{F} \Big[ f\mcl{H}[h]+\mcl{H}[f] h \Big](\nu)
= \mcl{F} \Big[ \mcl{H} \Big[ f\mcl{H}[h]+\mcl{H}[f]h \Big] \Big](\nu),
$$
and thus we arrive at the equality,
$$
\mcl{F} \Big[ \mcl{H}[f]\mcl{H}[h] - fh \Big](\nu)
=\mcl{F} \Big[ \mcl{H} \Big[ f\mcl{H}[h]+\mcl{H}[f]h \Big] \Big](\nu).
$$
Taking the inverse Fourier transform of this equation yields the result. 
For the general case $1/p+1/q \leq 1$, see the proof 
given in \cite{cartonlebrun77}. 
\end{proof}
%%%%%%%%%%%%%%%%%%%%%%%%%%%%%%%%%%%
Setting $f=h$ gives the following.
%%%%%%%%%%%%%%%%%%%%%%%%%%%%%%%%%%%%%%%%%%%
\begin{crl}\label{crl:SimpleHilbertProduct}
For $f \in L^p(\mbb{R})$ with $p \geq 2$, 
\begin{equation}
\mcl{H} \Big[ f\mcl{H}[f] \Big](x)=\frac{1}{2} \Big\{ \mcl{H}[f](x)^2-f(x)^2 \Big\}.
\label{eq:simple}
\end{equation}
\end{crl}

%%%%%%%%%%%%%%%%%%%%%%%%%%%%%%%%%%%%%%%%%%%%%%%%
\section{Proof of Theorem \ref{thm:main}}\label{sec:proof}
%%%%%%%%%%%%%%%%%%%%%%%%%%%%%%%%%%%%%%%%%%%%%%%%

In this section, we prove Theorem \ref{thm:main}.

\begin{proof}[Proof of {\rm (i)}]
For any $t \in [0, \infty)$, the Green's function $G(t, \cdot)$ is analytic
in $\mbb{C} \setminus \mbb{R}$.
Then the following equation is guaranteed by \eqref{eq:G2}, 
$$
\frac{\partial G(t, x+\sqrt{-1} y)}{\partial t} 
+G(t, x+\sqrt{-1} y) \frac{\partial G(t, x+\sqrt{-1} y)}{\partial x} =0,
\quad x \in \mbb{R}, \quad y \in \mbb{R} \setminus \{0\}.
$$
Put $y=\varepsilon >0$ and take the limit $\varepsilon \downarrow 0$.
By Proposition \ref{prp:HilbertAvgResolvent}, we will obtain
\begin{eqnarray}
&& 
\frac{\partial}{\partial t}
\Big\{ \mcl{H}[\rho(t, \cdot)](x) - \sqrt{-1} \rho(t, x) \Big\}
\nonumber\\
&& 
+ \pi \Big\{ \mcl{H}[\rho(t, \cdot)](x) - \sqrt{-1} \rho(t, x) \Big\}
\frac{\partial}{\partial x}
\Big\{ \mcl{H}[\rho(t, \cdot)](x) - \sqrt{-1} \rho(t, x) \Big\} =0.
\nonumber
\end{eqnarray}
This is written as
\begin{eqnarray}
\label{eq:eq1}
&& \left\{ \mcl{H} \left[ \frac{\partial \rho(t, \cdot)}{\partial t} \right](x)
+ \pi \frac{\partial}{\partial x} 
\frac{1}{2} \Big\{
\mcl{H}[\rho(t, \cdot)](x)^2-\rho(t,x)^2 \Big\} \right\}
\\
&& - \sqrt{-1}
\left\{ \frac{\partial \rho(t, x)}{\partial t}
+ \pi \frac{\partial}{\partial x}
\Big\{ \rho(t, x) \mcl{H}[\rho(t, \cdot)](x) \Big\} \right\}
=0.
\nonumber
\end{eqnarray}
Applying Corollary \ref{crl:SimpleHilbertProduct}
and Proposition \ref{prp:HilbertDerivative}, we have
\begin{eqnarray}
&& \frac{\partial}{\partial x} 
\frac{1}{2} \Big\{
\mcl{H}[\rho(t, \cdot)](x)^2-\rho(t,x)^2 \Big\} 
= \frac{\partial}{\partial x} \mcl{H} \Big[ \rho(t, \cdot) \mcl{H}[\rho(t, \cdot)] \Big](x)
\nonumber\\
&& \qquad \qquad 
= \mcl{H} \left[ 
\frac{\partial}{\partial x}
\Big\{\rho(t, \cdot) \mcl{H}[\rho(t, \cdot)] \Big\} \right](x)
\nonumber
\end{eqnarray}
Therefore, \eqref{eq:eq1} is equivalent with
$$
\mcl{H}[A(t, \cdot)](x) - \sqrt{-1} A(t, x)=0
\quad \Longleftrightarrow \quad
\mbox{$\mcl{H}[A(t, \cdot)](x)=0$ and $A(t, x)=0$}, 
$$
with
$$
A(t, x)=\frac{\partial \rho(t, x)}{\partial t}
+\frac{\partial}{\partial x}
\Big\{ \rho(t, x) \mcl{H}[\rho(t, \cdot)](x) \Big\} \in \mbb{R},
$$
for $t \in [0, \infty), x \in \mbb{R}$.
Since the Hilbert transform is invertible
by Proposition \ref{prp:InverseHilbert},
the equation \eqref{eq:rho2} is obtained.
\end{proof}

\begin{proof}[Proof of {\rm (ii)}] We apply the {\it method of characteristics}
\cite{AGZ10,blaizotnowak08,blaizotnowak10,forrestergrela15}
to solve the partial differential equation \eqref{eq:G2}.
We consider a parametrization $(t,z)=(t(r), z(r))$ such that $G(t(r), z(r))$ is constant for 
$r \in [0, \infty)$. 
That is, we construct a differentiable curve in $[0, \infty) \times \mbb{C} \setminus \mbb{R}$ 
along which $d G(t(r), z(r))/d r=0$. This construction leads to
%\begin{equation}\label{eq:CharacteristicParametrization}
$$
\frac{d t}{d r}\frac{\partial G(t, z)}{\partial t}
+ \frac{d z}{d r}\frac{\partial G(t, z)}{\partial z}
=0.
%\end{equation}
$$
Comparing this equation to \eqref{eq:G2} gives
\begin{equation}
\frac{d t}{d r}=1,\quad \frac{d z}{d r}=G(t, z).
\end{equation}
We can derive the explicit form of the parametrization from these equations. 
Clearly, $t$ and $r$ differ only by a constant, so we set $r=t-t_0 \geq 0$, 
and use $t$ as the parametrization variable. 
For $z$ we have
$$
\frac{d z}{d t}=G(t, z)=G(t_0, z_0).
$$
The second equality is derived from the requirement that $G(t, z)$ be constant along the curve $(t, z(t))$, 
where $z_0$ is the value of $z$ at $t=t_0$. Integrating this equation yields
\begin{equation}\label{eq:Characteristics}
z=tG(t_0, z_0)+z_0.
\end{equation}
Thus we obtain the equalities
$$
G(t, z)=G\Big(t_0, z-tG(t_0, z_0) \Big)=G\Big(t_0, z-t G(t, z) \Big).
$$
Without loss of generality, we can choose $t_0=0$ to obtain 
\eqref{eq:functional}.
\end{proof}

%%%%%%%%%%%%%%%%%%%%%%%%%%%%%%%%%%%%%%%%%%%%%%%%
\section{Solutions of hydrodynamic equations for special initial configurations}\label{sec:solutions}
%%%%%%%%%%%%%%%%%%%%%%%%%%%%%%%%%%%%%%%%%%%%%%%%
%%%%%%%%%%%%%%%%%%%%%%%%%%%%%%%%%%%%%%%%%%
\subsection{Case with one source}
%%%%%%%%%%%%%%%%%%%%%%%%%%%%%%%%%%%%%%%%%%

We consider the case where all the particles are located at a single point when $t=0$. 
Without loss of generality, we can choose the origin as the starting point, i.e.,
$$
\mu(\cdot)=\delta_0(\cdot)
\quad \Longleftrightarrow \quad
d \mu(x)=\rho(0, x) dx=\delta(x) dx,
$$
where $\delta(x)$ denotes Dirac's delta function.
Consequently,
\begin{equation}
G(0, z)=\int_{\mbb{R}} \frac{\rho(x,0) dx}{z-x}=\int_{\mbb{R}} \frac{\delta(x) dx}{z-x} =\frac{1}{z}.
\label{eq:G_0z_1}
\end{equation}
Then, \eqref{eq:functional} becomes
$$
t G^2(t, z)-zG(t, z)+1=0.
$$
This algebraic equation for $G(t, z)$ is solved by
\begin{equation}
G(t, z)=\frac{1}{2t} \left( z\pm\sqrt{z^2-4t} \right)=\frac{1}{2t}
\left\{ z\pm\sqrt{(z-2\sqrt{t})(z+2\sqrt{t})} \right\},
\quad t >0. 
\label{eq:G_tz_1}
\end{equation}
By the formula \eqref{eq:rho_G},
$$
\rho(t, x)=-\frac{1}{2\pi t} \Im \left\{ x\pm\lim_{\varepsilon \downarrow 0}
\sqrt{(x+\sqrt{-1} \varepsilon-2\sqrt{t})(x+\sqrt{-1} \varepsilon+2\sqrt{t})} \right\},
$$
but it should be that 
%$\lim_{x\to\pm\infty}\rho(t, x)=0$ and 
$\rho(t, x) \geq 0$, 
so we choose the lower (minus) sign.
By taking the limit we obtain
$$
\rho(t, x)=\begin{cases}
\displaystyle{\frac{1}{2 \pi t} \sqrt{4t-x^2}}, &\mnm{if }|x|<2\sqrt{t},\\
0, &\mnm{if }|x|\ge 2 \sqrt{t},
\end{cases}
$$
for $t > 0$. 
This is the time-dependent version of {\it Wigner's semicircle law} \cite{AGZ10,katori16}. 

Note that the two edges of the spectrum, $x=\pm2\sqrt{t}$, coincide with the conditions 
under which the method of characteristics breaks down
for the real characteristics.
We fix $t \in [0, \infty)$ and consider a map from $x_0 \in \mbb{R}$ to
$$
M_t(x_0):=\lim_{\varepsilon \downarrow 0}
t G(0, x_0+\sqrt{-1} \varepsilon) + x_0 \in \mbb{R}.
$$
If this map is injective for a domain $\mcl{D}_t \in \mbb{R}$, 
the real characteristic curves in the set
$\{ (s, M_s(x_0))_{0 \leq s \leq t} : x_0 \in \mcl{D}_t \}$
do not cross and to each of them corresponds a distinct value of $G$; 
$G(s, M_s(x_0))=G(0, x_0),$ $0 \leq s \leq t$.
That is, the method of characteristics works.
In the present case, \eqref{eq:G_0z_1} gives
$$
M_t(x_0)=\frac{t}{x_0}+x_0.
$$
The value of $x_0$ at which the method of characteristics breaks down,
denoted by $x_{0, \mnm{c}}$, is found by
$$
\frac{d M_t}{dx_0} (x_{0, \mnm{c}})
=-\frac{t}{x_{0, \mnm{c}}^2}+1=0
\quad \Longleftrightarrow \quad
x_{0, \mnm{c}}=\pm \sqrt{t}.
$$
The domain on $\mbb{R}$ for which the method of characteristics works
is given by
$\mcl{D}_t=\{x_0 \in \mbb{R} : |x_0| \geq |x_{0, \mnm{c}}| =\sqrt{t} \}$.
Actually, if $x_0 \in \mcl{D}_t$, \eqref{eq:G_tz_1} with the lower (minus) sign gives
\begin{eqnarray}
G(s, M_s(x_0)) &=& \frac{1}{2s} \left\{
\frac{s}{x_0}+x_0- \sqrt{\left( \frac{s}{x_0}-x_0 \right)^2 } \right\}
\nonumber\\
&=& \frac{1}{x_0}=G(0, x_0), \quad 0 \leq s \leq t.
\nonumber
\end{eqnarray}
The image of $\mcl{D}_t$ is given by
$$
\mcl{I}_t = \{ M_t(x_0) : x_0 \in \mcl{D}_t \}
=\Big\{x \in \mbb{R} : |x| \geq |M_t(x_{0, \mnm{c}})| = 2 \sqrt{t} \Big\},
$$
and the equality
$\mnm{supp}[\rho(t, \cdot)]=\mbb{R} \setminus \mcl{I}_t$
is established, $t \in [0, \infty)$. 
See Figure~\ref{fig:OneSourceBreakdown}.
\begin{figure}
\begin{center}
\begin{tikzpicture}
\draw [<->] (0,4) -- (0,0) -- (4,0);
\draw (-4,0) -- (0,0);
\draw (-0.1,2) -- (0.1,2);
\draw (2,-0.1) -- (2,0.1);
\draw (-2,-0.1) -- (-2,0.1);
\node [left] at (0,2) {2};
\node [below] at (2,0) {2};
\node [below] at (-2,0) {-2};
\draw (2,0) -- (4,4);
\draw (1.75,0) -- (3.5,49/16);
\draw (1.5,0) -- (3,2.25);
\draw (1.25,0) -- (2.5,25/16);
\draw (1,0) -- (2,1);
\draw (0.75,0) -- (3/2,9/16);
\draw (0.5,0) -- (1,0.25);
\draw (0.25,0) -- (1/2,1/16);
\draw (-2,0) -- (-4,4);
\draw (-1.75,0) -- (-3.5,49/16);
\draw (-1.5,0) -- (-3,2.25);
\draw (-1.25,0) -- (-2.5,25/16);
\draw (-1,0) -- (-2,1);
\draw (-0.75,0) -- (-3/2,9/16);
\draw (-0.5,0) -- (-1,0.25);
\draw (-0.25,0) -- (-1/2,1/16);
\draw [thick,domain=-4:4] plot (\x, {(\x/2)*(\x/2)});
\node [right] at (0,4) {$t$};
\node [below] at (4,0) {$z$};
\end{tikzpicture}
\end{center}
\caption{\label{fig:OneSourceBreakdown}Examples of real characteristics of the one-source case (straight lines) 
and limit curve for the support of $\rho(t,\cdot)$ (thick curve).}
\end{figure}
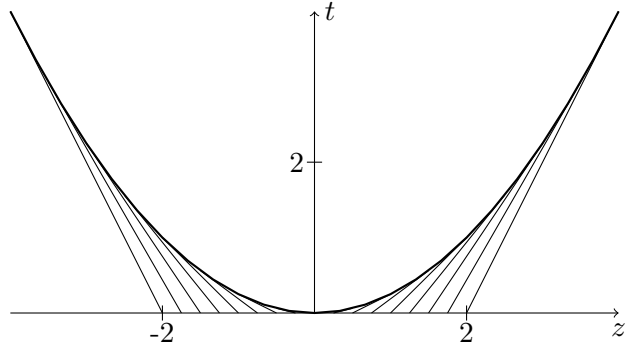

We put these observations in a more formal context in the following proposition.

%%%%%%%%%%%%%%%%%%%%%%%%%%%%%%%%%%%%%%%%%
\begin{prp}\label{prp:SuppRealChars}
At each $t \in [0, \infty)$ 
define
$$
\mcl{D}_t := \{ \mbox{$x_0 \in \mbb{R} : x=M_t(x_0)$ gives an injection to $\mbb{R}$} \}, 
\quad
\mcl{I}_t := \{ M_t(x_0) : x_0 \in \mcl{D}_t \}.
$$
Then for $t>0$,
\begin{equation}
\mnm{supp}[\rho(t, \cdot)] = \mbb{R}\setminus \mcl{I}_t.
\label{eq:statement}
\end{equation}
\end{prp}
%%%%%%%%%%%%%%%%%%%%%%%%%%%%%%%%%
\begin{proof}
Assume that $x \in \mcl{I}_t \subset \mbb{R}$.
Then the method of characteristics works and
$$
^{\exists}x_0 \in \mcl{D}_t \quad \mbox{s.t.} \quad
x=\lim_{\varepsilon \downarrow 0}
t G(0, x_0+\sqrt{-1} \varepsilon) + x_0.
$$
This implies that $\displaystyle{\lim_{\varepsilon \downarrow 0}
G(t, x+\sqrt{-1} \varepsilon) \in \mbb{R}}$.
By the equality \eqref{eq:Hilbert_rho} in Proposition \ref{prp:HilbertAvgResolvent},
$\rho(t, x)=0$ is concluded, that is,
$x \in \mbb{R} \setminus \mnm{supp}[\rho(t, \cdot)]$.

Conversely, assume that $x\notin \mnm{supp}[\rho(t, \cdot)]$.
Then, $\rho(t,x)=0$ and $\lim_{\varepsilon \downarrow 0}G(t, x+\sqrt{-1} \varepsilon)=G(t, x)\in\mbb{R}$
by Proposition~\ref{prp:HilbertAvgResolvent}.
Consider now the mapping
\begin{equation}
M^-_s(x):=x-sG(t, x),\ 0<s\leq t.
\end{equation}
By equation \eqref{eq:G2}, it follows that 
$$
\frac{d}{d s}G(t-s, M^-_s(x))\Big|_{s=0}=0.
$$
Therefore, $G(t-s, M^-_s(x))$ 
remains constant along the line $\{(t-s,M^-_s(x)):0<s\leq t\}$.
Note that $M^-_s(x)$ is an injective mapping, because
$$
\frac{dM^-_s(x)}{d x}=1-s\frac{d}{d x}G(t,x)=1+s\int_{\mnm{supp}[\rho(t,\cdot)]}\frac{\rho(t,u)}{(x-u)^2}du>0.
$$
This means that the method of characteristics works for $\{(t-s,M^-_s(x)):0<s\leq t\}$,
so we set $x_0=M^-_t(x)$ and find that
$$
G(t,x)=G(0,x_0),\ \mnm{and}\ x=x_0+tG(0,x_0).
$$
Consequently, $x\in\mcl{I}_t$ and the statement \eqref{eq:statement} is proved.
\end{proof}

%%%%%%%%%%%%%%%%%%%%%%%%%%%%%%%%%%%%%%%%
\subsection{Case with two sources}
%%%%%%%%%%%%%%%%%%%%%%%%%%%%%%%%%%%%%%%%

Now we consider the case where, for $a>0$,
$$
\mu(\cdot)=\frac{1}{2} \Big\{ \delta_{-a}(\cdot)+\delta_a(\cdot) \Big\}
\, \, \Longleftrightarrow \, \, 
d \mu(x)= \rho(0, x)dx =\frac{1}{2} \Big\{ \delta(x-a)+\delta(x+a) \Big\} dx.
$$
Then, the initial condition for $G(t, z)$ is given by
$$
G(0, z)=\frac{1}{2} \left( \frac{1}{z-a}+\frac{1}{z+a} \right),
$$
and \eqref{eq:functional} becomes
\begin{equation}\label{eq:TwoSourceDynG}
t^2G^3(t, z)-2ztG^2(t, z)+(z^2-a^2+t)G(t, z)-z=0.
\end{equation}
Note that, if we set $\tau:=t/a^2$ and $w:=z/a$ with $\bar G(\tau, w):=a G(a^2 \tau, a w)$, 
\eqref{eq:TwoSourceDynG} is transformed into
\begin{equation}\label{eq:NormTwoSourceDynG}
\tau^2\bar{G}^3(\tau, w)-2w\tau \bar{G}^2(\tau, w)+(w^2-1+\tau)\bar{G}(\tau, w)-w=0.
\end{equation}
Therefore, without loss of generality we solve \eqref{eq:TwoSourceDynG} for $a=1$ 
and assume that time and space are given in units of $a^2$ and $a$, respectively.

Before we solve \eqref{eq:TwoSourceDynG}, we use Proposition~\ref{prp:SuppRealChars} 
to determine the support of the particle density. The parametrization equation for the real characteristics is given by
$$
M_t(x_0)=\frac{t}{2} \left(\frac{1}{x_0-1}+\frac{1}{x_0+1} \right)+x_0=x_0 \left(1-\frac{t}{1-x_0^2} \right),
\quad t \in [0, \infty).
$$
The breakdown condition for the injective map from $x_0$ to $x=M_t(x_0)$ with $t$ fixed is
\begin{equation}\label{eq:TwoSourceCriticalCondition}
\frac{d M_t}{dx}(x_{0, \mnm{c}})=0
\quad \Longleftrightarrow \quad
t=\frac{(1-x_{0, \mnm{c}}^2)^2}{1+x_{0, \mnm{c}}^2}.
\end{equation}
We show the plot of several characteristics and the breakdown curve in Figure~\ref{fig:TwoSourceBreakdown}.
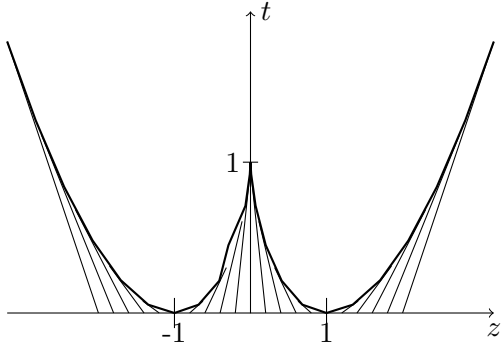
\begin{figure}
\begin{center}
\begin{tikzpicture}[yscale=2]
\draw [<->] (0,2) -- (0,0) -- (3.2,0);
\draw (-3.2,0) -- (0,0);
\draw (-0.1,1) -- (0.1,1);
\draw (1,-0.1) -- (1,0.1);
\draw (-1,-0.1) -- (-1,0.1);
\node [left] at (0,1) {1};
\node [below] at (1,0) {1};
\node [below] at (-1,0) {-1};
\draw [thick,domain=-2:2] plot ({\x*(\x*\x-1)/(1+\x*\x)+\x},{(1-\x*\x)*(1-\x*\x)/(1+\x*\x)});
\draw [domain=0:(1-0.2*0.2)*(1-0.2*0.2)/(1+0.2*0.2)] plot ({0.2*(1-\x/(1-0.2*0.2))},\x);
\draw [domain=0:(1-0.4*0.4)*(1-0.4*0.4)/(1+0.4*0.4)] plot ({0.4*(1-\x/(1-0.4*0.4))},\x);
\draw [domain=0:(1-0.6*0.6)*(1-0.6*0.6)/(1+0.6*0.6)] plot ({0.6*(1-\x/(1-0.6*0.6))},\x);
\draw [domain=0:(1-0.8*0.8)*(1-0.8*0.8)/(1+0.8*0.8)] plot ({0.8*(1-\x/(1-0.8*0.8))},\x);
\draw [domain=0:(1-1.2*1.2)*(1-1.2*1.2)/(1+1.2*1.2)] plot ({1.2*(1-\x/(1-1.2*1.2))},\x);
\draw [domain=0:(1-1.4*1.4)*(1-1.4*1.4)/(1+1.4*1.4)] plot ({1.4*(1-\x/(1-1.4*1.4))},\x);
\draw [domain=0:(1-1.6*1.6)*(1-1.6*1.6)/(1+1.6*1.6)] plot ({1.6*(1-\x/(1-1.6*1.6))},\x);
\draw [domain=0:(1-1.8*1.8)*(1-1.8*1.8)/(1+1.8*1.8)] plot ({1.8*(1-\x/(1-1.8*1.8))},\x);
\draw [domain=0:(1-2*2)*(1-2*2)/(1+2*2)] plot ({2*(1-\x/(1-2*2))},\x);
\draw [domain=0:(1-0.2*0.2)*(1-0.2*0.2)/(1+0.2*0.2)] plot ({-0.2*(1-\x/(1-0.2*0.2))},\x);
\draw [domain=0:(1-0.4*0.4)*(1-0.4*0.4)/(1+0.4*0.4)] plot ({-0.4*(1-\x/(1-0.4*0.4))},\x);
\draw [domain=0:(1-0.6*0.6)*(1-0.6*0.6)/(1+0.6*0.6)] plot ({-0.6*(1-\x/(1-0.6*0.6))},\x);
\draw [domain=0:(1-0.8*0.8)*(1-0.8*0.8)/(1+0.8*0.8)] plot ({-0.8*(1-\x/(1-0.8*0.8))},\x);
\draw [domain=0:(1-1.2*1.2)*(1-1.2*1.2)/(1+1.2*1.2)] plot ({-1.2*(1-\x/(1-1.2*1.2))},\x);
\draw [domain=0:(1-1.4*1.4)*(1-1.4*1.4)/(1+1.4*1.4)] plot ({-1.4*(1-\x/(1-1.4*1.4))},\x);
\draw [domain=0:(1-1.6*1.6)*(1-1.6*1.6)/(1+1.6*1.6)] plot ({-1.6*(1-\x/(1-1.6*1.6))},\x);
\draw [domain=0:(1-1.8*1.8)*(1-1.8*1.8)/(1+1.8*1.8)] plot ({-1.8*(1-\x/(1-1.8*1.8))},\x);
\draw [domain=0:(1-2*2)*(1-2*2)/(1+2*2)] plot ({-2*(1-\x/(1-2*2))},\x);
\node [right] at (0,2) {$t$};
\node [below] at (3.2,0) {$z$};
\end{tikzpicture}
\end{center}
\caption{\label{fig:TwoSourceBreakdown}Examples of real characteristics of the two-source case (straight lines)
and limit curve for the support of $\rho(t,\cdot)$ (thick curve). 
Note that the support becomes connected at $t=1$.}
\end{figure}
We observe that when $t\in[0,1)$, the support of the particle density is disjoint, and it is given by the expression
$$
\mnm{supp}[\rho(t, \cdot)]=\left\{ x\in\mbb{R}: \sqrt{B_-(t)} < |x| < \sqrt{B_+(t)} \right\},
$$
where
$$
B_\pm(t)=\Big[1+\frac{1}{2}\Big(t\pm\sqrt{t(t+8)}\Big)\Big]\Big[1-\frac{1}{4}\Big(t\mp\sqrt{t(t+8)}\Big)\Big]^2.
$$
Furthermore, when $t \in (1,\infty)$ its support is connected, and it is given by
$$
\mnm{supp}[\rho(t, \cdot)]=\left\{x\in\mbb{R}: 0 \leq |x| <\sqrt{B_+(t)} \right\}.
$$
Indeed, the two disjoint parts of the support join at $(t, x)=(1, 0)$. 
That they join at $x=0$ is clear from the symmetry of the system, 
and by \eqref{eq:TwoSourceCriticalCondition}, 
setting $x_{0, \mnm{c}}=0$ means that the time at which the support becomes connected is $t=1$.

The solutions to \eqref{eq:TwoSourceDynG} with $a=1$ are given by
$$
G_n(t, z)=\frac{1}{3 t^2}\Big[2zt - C_n(t, z) - \frac{t^2[z^2-3(t-1)]}{C_n(t, z)}\Big], \quad n=0,1,2,
$$
where $C_n(t, z),$ $n=0,1,2$ are the three complex cubic roots of
$$
C^3(t, z)=t^3
\left[\sqrt{27[(t-1)^3+z^2(2+5t-t^2/4)-z^4]}- \frac{z}{2} \{ 9(t+2)-2z^2 \} \right].
$$
We choose $C_n(z,t)=\mnm{e}^{2 \pi \sqrt{-1} n/3}C(t, z)$, $n=0,1,2$, 
where $C(t, z)$ is the particular cubic root of $C^3(t, z)$ 
taken so that $\sqrt[3]{-1}=\mnm{e}^{\pi \sqrt{-1}/3}$. 

Now we use the formula \eqref{eq:rho_G}.
For $t \in (0, 1),$ $x< 0$ and for $t \in [1, \infty),$ $x< \sqrt{3(t-1)}$, let
$$
\rho_\mnm{l}(t, x):=\frac{1}{2\sqrt{3}\pi t^2} \left[C_0(t, x)-\frac{t^2 \{x^2-3(t-1)\}}{C_0(t, x)} \right].
$$
In this case, $C^3(t, x)$ is positive, so $C_0(t, x) >0$. 
For $t \in (0, 1),$ $x \geq 0$ and for $t \in [1, \infty),$ $x \geq \sqrt{3(t-1)}$, let
$$
\rho_\mnm{r}(t, x):=\frac{1}{2\sqrt{3}\pi t^2}\left[C_1(t, x)-\frac{t^2 \{x^2-3(t-1)\} }{C_1(t, x)} \right].
$$
In this case, $C^3(t, x) <0$, and its cubic root has a phase of $\pi/3$. 
By choosing $n=1$, we add a phase of $2\pi/3$ 
to obtain a negative cubic root of $C^3(t, x)$ 
and obtain the positive function $\rho_\mnm{r}$.
The density function is determined as follows.\\
When $t \in (0, 1)$, 
$$
\rho(t, x) \nonumber\\
=\begin{cases}
\rho_\mnm{l}(t,x), 
&\mnm{if } - \sqrt{B_+(t)} < x < - \sqrt{B_-(t)}, \\
\rho_\mnm{r}(t, x),
&\mnm{if } \sqrt{B_-(t)} < x < \sqrt{B_+(t)}, \\
0,
&\mnm{if } 0 \leq |x| \leq \sqrt{B_-(t)} \quad \mnm{ or } |x| \geq \sqrt{B_+(t)}, 
\end{cases}
$$
and when $t \in (1, \infty)$, 
$$
\rho(t, x) \nonumber\\
=\begin{cases}
\rho_\mnm{l}(t,x), 
&\mnm{if } - \sqrt{B_+(t)} < x < \sqrt{3(t-1)}, \\
\rho_\mnm{r}(t, x),
&\mnm{if } \sqrt{3(t-1)} < x < \sqrt{B_+(t)}, \\
0,
&\mnm{if } |x| \geq \sqrt{B_+(t)}. 
\end{cases}
$$
The profiles of $\rho(t, \cdot)$ are plotted 
for several values of time in Figure~\ref{fig:TwoSourcesDynDensity}.
%%%%%%%%%%%%%%%%%%%%%%%%%%%%%%%%%%%%%%%%%%%%%%%%
\begin{figure}
\begin{center}
\includegraphics[width=0.5\textwidth]{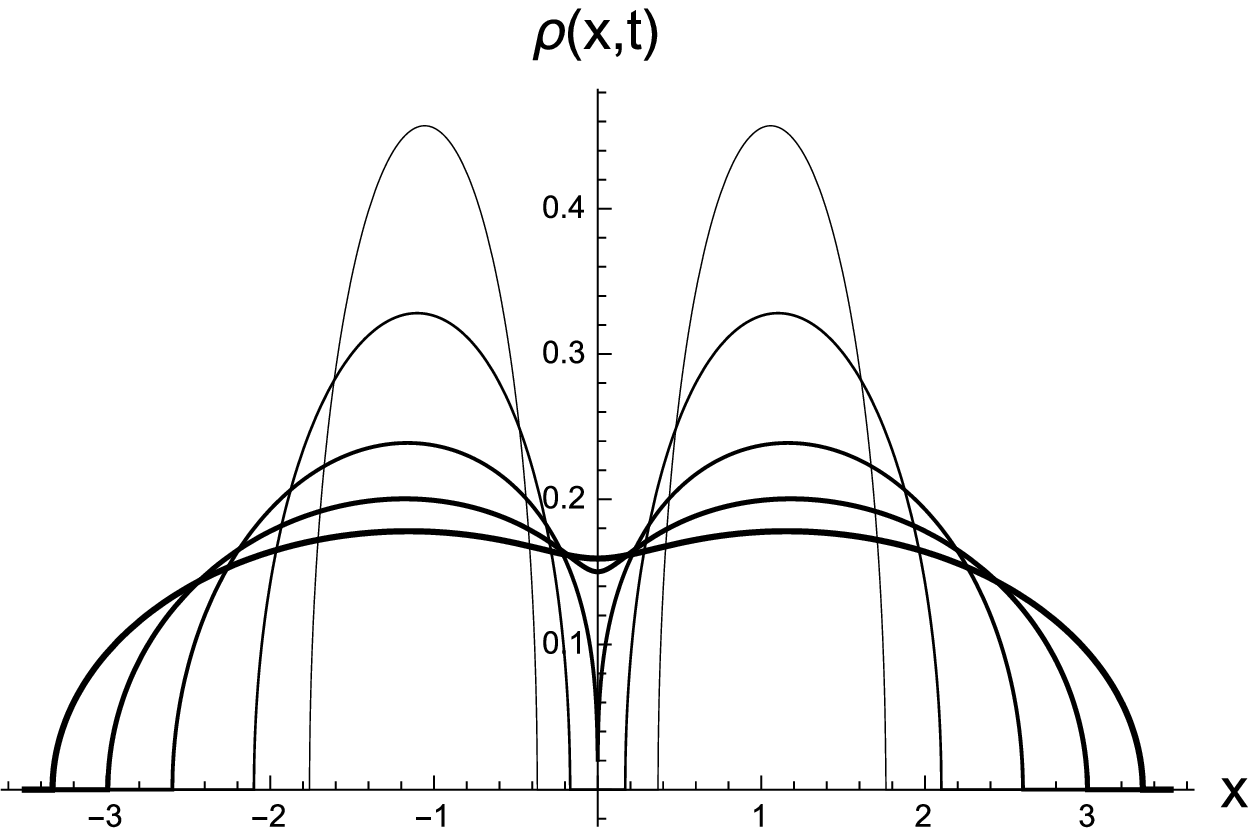}
\end{center}
{\caption{\label{fig:TwoSourcesDynDensity}}
Density functions with the two sources
at $x=\pm 1$ for $t=0.25$ (the thinnest curve), 0.5, 1, 1.5 
and 2 (the thickest curve).}
\end{figure}
%%%%%%%%%%%%%%%%%%%%%%%%%%%%%%%%%%%%%%%%%%%%%%%%%

The density at time $t=1$ is of particular interest, as it is the the point in time where the support of $\rho(t, \cdot)$ 
becomes connected after being disjoint for $t \in (0, 1)$. For that case, \eqref{eq:NormTwoSourceDynG} becomes
$$
G^3(1, z)-2zG^2(1, z)+z^2G(1, z)-z=0.
$$
From our previous considerations, it can be shown that the above solution
gives
\begin{eqnarray}
&& \rho(1, x) \nonumber\\
&& =\begin{cases}
\displaystyle{
\frac{3}{4 \pi}\left(\frac{2|x|}{3\sqrt{3}} \right)^{1/3}
\left[\left(1+\sqrt{1-\frac{4 x^2}{27}} \right)^{2/3}-\left(1-\sqrt{1-\frac{4 x^2}{27}} \right)^{2/3} \right]
}, 
&\mnm{if }\displaystyle{|x|<\frac{3 \sqrt{3}}{2}},\\
0,  &\mnm{if }\displaystyle{|x| \geq \frac{3 \sqrt{3}}{2}}.
\end{cases}
\nonumber
\end{eqnarray}
This coincides with the density function given as (6.118) in Section 6.5 
of \cite{nadalthesis} if we set $L=3\sqrt{3}/2$. 
This form shows how the disjoint parts of the density meet at the origin when $t=1$, 
because the density behaves as $\sqrt{3}|x|^{1/3}/2\pi$ for small $x$. 
See also Section 2.3 of \cite{warcholthesis}. 

\begin{acknowledgments}
MK is supported in part by the Grant-in-aid for Scientific Research (C)
(No.26400405) of the Japan Society for the Promotion of Science.
\end{acknowledgments}

%%%%%%%%%%%%%%%%%%%%%%%%%%%%%%%%%%%%%%%%%%%%%%%%%


\begin{thebibliography}{99}
%
% The \bibitem commands: 
% Please follow "Notice to Authors" for referencing.  You 
% must specify bold and italic fonts yourself. 
%
\bibitem{AGZ10}
Anderson, G. W., Guionnet, A., Zeitouni, O., 
\textit{An Introduction to Random Matrices}, 
Cambridge University Press, Cambridge (2010).

\bibitem{blaizotnowak08} 
Blaizot, J.-P., Nowak, M. A., 
Large-$N_c$ confinement and turbulence,
\textit{Phys. Rev. Lett.}, 
\textbf{101} (2008), 102001/1-4.

\bibitem{blaizotnowak10} 
Blaizot, J.-P., Nowak, M. A., 
Universal shocks in random matrix theory,
\textit{Phys. Rev. E} 
\textbf{82} (2010), 051115/1-6.

\bibitem{cartonlebrun77} 
Carton-Lebrun, C., 
Product properties of Hilbert transforms,
\textit{J. Approx. Theor.}, 
\textbf{21} (1977), 356--360.

\bibitem{dyson62} 
Dyson, F. J., 
A Brownian-motion model for the eigenvalues of a random matrix,
\textit{J. Math. Phys.}, 
\textbf{3} (1962), 1191-1198.

\bibitem{forrester10}
Forrester, P. J., 
\textit{Log-Gases and Random Matrices}, 
Princeton University Press, Princeton (2010)

\bibitem{forrestergrela15} 
Forrester,~P. J., Grela, J., 
Hydrodynamical spectral evolution for random matrices,
\textit{arXiv:1507.07274}, 2015.

\bibitem{katori16}
Katori, M.,
\textit{Bessel Processes, Schramm--Loewner Evolution,
and the Dyson Model},
Springer (2016).

%\bibitem{katoritanemura04} 
%Katori, M., Tanemura, H.,
%Symmetry of matrix-valued stochastic processes and noncolliding diffusion particle systems,
%\textit{J. Math. Phys.}, \textbf{45} (2004), 3058--3085.

%\bibitem{mehta04} 
%Mehta, M.~L., 
%\textit{Random Matrices (3rd Ed.)}, Elsevier (2004).

\bibitem{nadalthesis} 
Nadal, C., 
\emph{Matrices al{\'e}atoires et leurs applications {\`a} la physique statistique et physique quantique},
PhD Thesis (2011), 
Universit{\'e} Paris-{S}ud {XI}.

\bibitem{warcholthesis} 
Warcho{\l}, P., 
\emph{Dynamic properties of random matrices - theory and applications},
PhD Thesis (2014), 
Jagiellonian University.

%
%\bibitem{d} DeVore,~R.~A., 
%Approximation of functions,
%\textit{Proc. Sympos. Appl. Math.,} vol. 36,
%Amer. Math. Soc., Providence, RI, 1986, pp. 34--56.
%
%
%\bibitem{kk}
%Kashiwara,~M. and Kawai,~T., 
%On the boundary value problem for
%elliptic systems of linear partial differential equations I-II,
%\textit{Proc. Japan. Acad.}, 
%\textbf{48} (1971), 712--715; 
%\textit{ibid.},
%\textbf{49} (1972), 164--168.
%
%\bibitem{ksww}
%Kalf.~H., Schmincke,~U.-W., Walter,~J. and W\"ust,~R.,
%On the spectral theory of Schr\"odinger and Dirac operators 
%with strongly singular potentials, 
%\textit{Proceedings of the
%Symposium on Spectral Theory and Differential Equation},
%University of Dundee, 1974, 
%\textit{Lecture Notes in Math.}, 
%Springer-Verlag, Berlin, Heidelberg, New York, \textbf{448} (1975),
%182--226.
%
%\bibitem{mko}
%Mori,~S. and Koll\'ar,~J.,  
%\textit{Birational Geometry of Algebraic Varieties},
%Cambridge University Press, 1998.
%
%\bibitem{mke}
%Mori,~S. and Keel,~S., 
%Quotients by groupoids,
%\textit{Ann. of Math.}, 
%\textbf{145} (1997), 193--213.

\end{thebibliography}
\end{document}